\documentclass[a4paper,11pt]{article}

\usepackage{vmargin}
\usepackage{mathrsfs}
\usepackage[sans]{dsfont}

\usepackage{amssymb}
\usepackage{amsmath}
\usepackage{amsthm}
\usepackage{tikz}
\usetikzlibrary{shapes,arrows}

\usepackage[utf8]{inputenc}
\newcommand{\N} {\mathbb{N}}
\newcommand{\R} {\mathbb{R}}
\newcommand{\C} {\mathbb{C}}
\newcommand{\e} {\varepsilon}

\newcommand{\di}[1]{\operatorname{d}\!#1}

\newtheorem{proposition}{Proposition}
\newtheorem{theorem}{Theorem}
\newtheorem{lemma}{Lemma}

\newtheorem{remark}{Remark}

\title{{\bf Controllability of ensembles of linear dynamical systems}}
\author{Michael Sch\"onlein and Uwe Helmke\footnote{(helmke,schoenlein)@mathematik.uni-wuerzburg.de}\\[2ex]
Institute for Mathematics\\
University of W\"urzburg\\
Emil-Fischer Stra\ss e 40\\
97074 W\"urzburg\\
Germany
}

\usepackage{a4wide}

\begin{document}

\maketitle

\begin{abstract}
We investigate the task of controlling ensembles of
initial and terminal state vectors of parameter-dependent linear
systems by applying  parameter-independent open loop controls. Necessary, as well as sufficient, conditions for ensemble
controllability are established, using tools from complex
approximation theory. For real analytic families of linear systems it is shown that ensemble
controllability holds only for systems with at most two
independent 
parameters. We apply the results to networks of linear systems and address the question of open-loop robust synchronization.\\[3ex]
\end{abstract}

%

\section{Introduction}\label{sec:intro}

Driven by recent engineering applications the task of controlling
ensembles of system by open loop controls has gained considerable
attraction. The motivation for this study originates, for instance, from quantum control \cite{li2006control}, the control of spatially invariant systems such as platoons \cite{bamieh2002distributed, Curtain2009}, the control of partial differential equations \cite{coron2007book}, and flocks of systems \cite{brockett2010control}.

From a functional analytic point of view, the problem of ensemble
controllability is equivalent to approximate controllability of infinite-dimensional
systems defined on Banach or Hilbert spaces~\cite{scl2014}. Standard
characterizations of approximate controllability in Hilbert spaces
can be found in \cite{curtainzwart95,fuhrmann1972weak}. However, the
results in \cite{curtainzwart95,Jacob2006diagonal} are, except for
very special cases, not easily applicable for ensemble control
as they depend on the existence of a Riesz basis of eigenvectors. A
new function theory approach to uniform ensemble
controllability has been developed in \cite{scl2014}, using classical
approximation theoretic results, such as the Stone-Weierstrass Theorem \cite{davis63} and Mergelyan's Theorem \cite{walsh65}.

In this paper, recent results on ensemble control by
\cite{scl2014,pamm2014,mtns2014} are extended 
in several directions. The necessary and sufficient conditions for
ensemble control in \cite{scl2014} are
extended to 
finite unions of disjoint parameter intervals. Somewhat surprisingly, we show that
ensemble controllability of real analytic families of systems holds
only if the parameter space is at most two-dimensional. New explicit
characterizations of uniform ensemble controllability are derived for a special class
of one-parameter families of systems. Finally, we discuss an application to 
robust synchronization for circular interconencted homogeneous
networks.  

\newpage

\section{Polynomial characterizations of uniform ensemble controllability}

We consider parameter-dependent linear time-invariant
systems, described in continuous-time by

\begin{equation}\label{eq:sys-cont}
 \begin{split}
 \Sigma(\theta):\,
 \begin{cases} 
  \tfrac{\partial}{\partial t}x(t,\theta)=A(\theta)x(t,\theta)+B(\theta)u(t)&\\
  x(0,\theta) = 0,&
\end{cases}
 \end{split}
\end{equation}
and in discrete-time, by
\begin{equation}\label{eq:sys-discrete}
 \begin{split}
  \Sigma(\theta):\,
 \begin{cases} 
  x_{t+1}(\theta)=A(\theta)x_t(\theta)+B(\theta)u_t&\\
  x_0(\theta) = 0.&
 \end{cases}
 \end{split}
\end{equation}
Since we want explore reachability properties of these systems, the
  initial state is taken to be zero. Note that in the continuous-time
  case there is no difference between reachability and
  controllability. In the sequel we will only use the term
  controllability, with the caveat that for discrete-time systems this
  has to be interpreted as reachability. Recall that a system
  (\ref{eq:sys-cont}), or (\ref{eq:sys-discrete}), 
is reachable for a fixed parameter $\theta$ if and only if the Kalman matrix has full rank, i.e. 
\begin{align*}
 \operatorname{rank} 
 \begin{pmatrix} B(\theta)& A(\theta)B(\theta) &  \cdots & A(\theta)^{n-1}B(\theta) \end{pmatrix}
=n.
\end{align*}
Throughout the paper we assume that the parameter $\theta$ varies in a compact subset
$\textbf{P} \subset \R^d$ and the system matrices
$A(\theta)\in\mathbb{R}^{n\times n}, B(\theta)\in \mathbb{R}^{n\times
  m}$ are continuous functions in $\theta$, i.e. $A \in C(\mathbf{P},\mathbb{R}^{n\times n})$ and $ B\in C(\mathbf{P},\mathbb{R}^{m \times n})$. 
We refer to such parameter-dependent families of systems and
initial/final  states as an \emph{ensemble}.  
Throughout this
paper we assume that $\textbf{P}$
is a nonempty compact subset of $\mathbb{R}^d$ with
$\overline{\operatorname{int}\,\mathbf{P}} =\mathbf{P}$. This implies
that the dimension of $\textbf{P}$ is well defined and satisfies $\dim
\textbf{P}=d$. It excludes the well-understood case of
parallel interconnected linear systems, where
$\textbf{P}$ is a finite set.\\ 

In this paper we address the following open loop control task for
parametric families (\ref{eq:sys-cont}) and (\ref{eq:sys-discrete}). 
Let 
\begin{align*}
 \varphi(T,\theta,u) = \int_0^T \operatorname{e}^{A(\theta)\,(T-s)} B(\theta) u(s) \di{s}
\end{align*}
and
\begin{align*}
 \varphi(T,\theta,u) = \sum_{k=0}^{T-1} A(\theta)^k B(\theta) u_{T-1-k},
\end{align*}
denote the solutions of (\ref{eq:sys-cont}) and (\ref{eq:sys-discrete}), 
respectively.  An ensemble $\Sigma(\theta)$ is called \emph{uniformly
  ensemble controllable}, if  for any $x^* \in C(\mathbf{P},\R^n)$ and
any $\e >0$ there exists a $T>0$ and an input function $u\in L^1([0,T],\mathbb{R}^m)$ (or, in
discrete-time,  a finite input sequence $u_0,...,u_{T-1}$) such that
\begin{align}\label{eq:openloop}
 \sup_{\theta \in \textbf{P}} \| \varphi(T,\theta,u) - x^*(\theta) \| < \e.
\end{align}

They key point to note is that the input function (or input sequence)
is assumed to be independent of the parameter values $\theta$. Thus
the open loop input function $u$ achieves (\ref{eq:openloop})
universally for all $\theta$.  
For the subsequent analysis the following necessary conditions are
important. For the proof we refer to
\cite[Lemma~1]{scl2014}. Let $\sigma(A)\subset \mathbb{C}$ denote the spectrum 
of an $n\times n$-matrix $A$, i.e., the set of real and complex
eigenvalues of $A$. 

\begin{proposition}\label{prop-nec}
Suppose that the ensemble (\ref{eq:sys-cont}) (or
 (\ref{eq:sys-discrete})) is uniform ensemble controllable. Then
\begin{enumerate}
\item[(E1)] For every $\theta\in \textbf{P}$ the linear system $(A(\theta),B(\theta))$ is reachable.
\item[(E2)] For each number $s\geq m+1$ of distinct parameters $\theta_1, \ldots, \theta_s
  \in \textbf{P}$, the spectra of $A(\theta)$ satisfy 
\begin{equation*}
\sigma (A(\theta_1))\cap \cdots \cap \sigma (A(\theta_s))=\emptyset.
\end{equation*}  %
\end{enumerate}
\end{proposition}
The preceding two
conditions are useful in order to rule out families that are not
ensemble controllable. In particular,  condition (E2) implies that
$A(\theta)$ cannot have a
$\theta$-independent eigenvalue. We next show that necessary and sufficient
conditions for uniform ensemble controllability can be stated in terms
of a polynomial approximation property. \\
  
For discrete-time ensembles with a scalar input sequence
$u_0,...,u_{T-1}$, the solution at $T>0$ is 
\begin{align*}
 \varphi(T,\theta,u)= \left(u_{T-1}I + u_{T-2}A(\theta)+ \cdots + u_{0}A(\theta)^{T-1} \right)b(\theta).
\end{align*}
This implies that a family of single-input, discrete-time systems $(A(\theta),b(\theta))$ is uniformly ensemble controllable if and only if for all $\e>0$ and all $x^* \in C(\mathbf{P},\R^n)$ there is a real scalar polynomial $p \in \R[z]$ such that 
\[
\sup_{\theta \in \mathbf{P}} \| p(A(\theta))\,b(\theta) - x^*(\theta) \| < \e.
\]
For a multivariable discrete-time system, the ensemble control
condition (\ref{eq:openloop}) is similarly seen as being equivalent 
to 
\[
\sup_{\theta \in \mathbf{P}} \| \sum_{j=1}^{m}p_j(A(\theta))\,b_j(\theta) - x^*(\theta) \| < \e.
\]

More generally, we obtain a polynomial characterization of ensemble
controllability for both continuous-time and discrete-time systems.

\begin{remark}\label{samplemma}
Let $\textbf{P}$ be compact. Assume that:
\begin{enumerate}
\item[(K1)] $A(\theta)$ has simple spectra for all $\theta$.

\item[(K2)] For all $\theta\neq \theta'$ the spectra of $A(\theta))$ and
  $A(\theta')$ are disjoint. 
\end{enumerate}
Then, the connected components of
\[K=\bigcup_{\theta\in\textbf{P}}^{}\sigma (A(\theta)) \; \subset \mathbb{C}\]
are simply connected. Note that the converse is false in general.
\end{remark}
\medskip

\noindent The next result shows that uniform ensemble controllability
for continuous-time and discrete-time systems (\ref{eq:sys-cont}) and
(\ref{eq:sys-discrete}) are equivalent.

\begin{theorem}\label{prop:1}
Let $A\in C(\textbf{P},\R^{n\times n})$ and $B\in
C(\textbf{P},\R^{n\times n})$. Let $b_1(\theta), \ldots,b_m(\theta)$
denote the columns of $B(\theta)$. The following assertions are equivalent. 
\begin{enumerate}
\item[(a)] The continuous-time ensemble (\ref{eq:sys-cont}) is
  uniformly ensemble controllable.

\item[(b)] The discrete-time ensemble (\ref{eq:sys-discrete}) is
  uniformly ensemble controllable.

\item[(c)] For each $\e>0$ and each $x^* \in C(\mathbf{P},\R^n)$
  there exist real scalar polynomials $p_1, \ldots, p_m \in \R[z]$ such that 
\begin{align}\label{eq:con-ensemble}
\sup_{\theta \in \mathbf{P}} \| \sum_{j=1}^{m}p_j(A(\theta))\,b_j(\theta) - x^*(\theta) \| < \e.
\end{align}
\item[(d)] The set 
\[
\operatorname{span}\{A(\theta)^kb_j(\theta)\;|\; k\in \mathbb{N}_0,
j=1, \ldots,m\}
\]
of continuous functions in $\theta$ is dense in $C(\textbf{P},\R^n)$
with respect to the sup-norm.
\end{enumerate}
\end{theorem}

\begin{proof}
The preceding remarks prove the equivalence of (a) and (c). The
equivalence of (c) and (d) is obvious and the
equivalence of (b) and (d) follows from \cite[Theorem~3.1.1]{triggiani1975}.
\end{proof}
\medskip


\section{Refined characterizations of uniform ensemble controllability}

The preceding results show that the problem of uniform ensemble
controllability is intimately connected to the classical problem of
approximating a continuous function by a polynomial. 
In the sequel, unless stated otherwise, all results hold for uniform ensemble controllability of both continuous-time and discrete-time
 systems. 
In this section we will derive more explicit necessary and sufficient
conditions. In particular, we will show that uniform controllability
can only hold for systems depending on at most two parameters.\medskip
 
We begin with a technical result. Given an array of linear
parameter dependent systems 
$(A_{ij},B_{ij})\in C(\textbf{P},\R^{n_i\times n_j}) \times
C(\textbf{P}
,\R^{n_i\times
m_j})$ with $1\leq i\leq j\leq N$. Let $\overline{n}=\sum_{j=1}^{N}n_j$
and $\overline{m}=\sum_{j=1}^{N}m_j$. Define the associated upper triangular
ensemble of systems by
\begin{equation}\label{eq:uppertriang}
 A= 
\begin{pmatrix}
 A_{11}  & \cdots &A_{1N}\\
 & \ddots &\vdots\\
 0& & A_{NN} 
\end{pmatrix}
\in C(\mathbf{P},\R^{\overline{n}\times \overline{n}})
,\quad 
B = 
\begin{pmatrix}
 B_{11}  & \cdots &B_{1N}\\
 & \ddots &\vdots\\
 0& & B_{NN} 
\end{pmatrix}
\in C(\mathbf{P},\R^{\overline{n}\times \overline{m}}).
\end{equation}
 
\begin{proposition}\label{lem:sirect-sum}
The upper triangular family of systems $(A(\theta), B(\theta))$ is uniform ensemble
controllable  if and only if the families $(A_{ii}( \theta),
B_{ii}(\theta))$ are uniform ensemble controllable for $i=1, \ldots,N$.
\end{proposition}

\begin{proof}
The necessity part is obvious. Thus assume that $(A_{ii},B_{ii})$ are
uniform ensemble controllable. By Theorem~\ref{prop:1} it is sufficient to consider continuous-time systems. For simplicity we focus on $N=2$, i.e., on
\begin{align*}
\dot{x}_1(t)&=A_{11}(\theta)x_1(t)+A_{12}(\theta)x_2(t)+B_{11}(\theta)u_1(t)+B_{12}(\theta)u_2(t)\\
\dot{x}_2(t)&=A_{22}(\theta)x_2(t)+B_{22}(\theta)u_2(t).
\end{align*}
The general case is
treated, proceeding by induction. Let $x_i(t)$ denote
the solution of the $i$-th component.
Given $x^*=\operatorname{col}(x^*_1\,\cdots \, x^*_N) \in
C(\mathbf{P},\R^{\overline{n}})$ and let $\e>0$. By uniform ensemble
controllability of $(A_{22},B_{22})$ there exists an input $u_2\in
L^1([0,T],\R^{m_2})$ such that 
\[
\sup_{\theta \in \textbf{P}} \| x_2(T) - x_2^*(\theta) \| < \e.
\]

Let $u=\operatorname{col}(u_1\,\cdots\,u_N) \in L^1([0,T],\R^{\overline{m}})$ with $u_i
\in L^1([0,T],\R^{m_i})$. 
Applying the variations of constant formula we have
\[
x_2(t)=\int_{0}^{t}e^{(t-s)A_{22}}B_{22}u_2(s)\di{s}
\]
and thus
\begin{align*}
x_1(T)=z_1(T) + \int_{0}^{T}e^{(T-s)A_{11}}B_{11}u_1(s)\di{s},
\end{align*}
where
\[
z_1(T)=\int_{0}^{T}e^{(T-s)A_{11}}
\left(
A_{12}\int_{0}^{s}e^{(s-\tau)A_{22}}B_{22}u_2(\tau)\di{\tau} +B_{12}u_2(s)
\right) \di{s}.
\]
By uniform ensemble controllability of $(A_{11},B_{11})$ there exists
an input  $u_1\in
L^1([0,T],\R^{m_1})$ with
 \[
\sup_{\theta \in \textbf{P}} \| \int_{0}^{T}e^{(T-s)A_{11}}B_{11}u_1(s)\di{s} - x_1^*(\theta)+z_1(T) \| < \e.
\]
But this implies 
\[
\sup_{\theta \in \textbf{P}} \| x_1(T) - x_1^*(\theta) \| < \e
\]
and we are done.
 \end{proof}

Using an appropriate similarity transformation, every system can be
transformed into Hermite canonical form \cite{kailath1980,linnemann1984bestimmung}, which has the upper triangular form
(\ref{eq:uppertriang}).  Given a matrix pair $(A,B) \in
\R^{n\times n} \times \R^{n\times m}$, where $b_i$ is the $i$th column
of $B$. Select from left to right in the permuted Kalman matrix 
\begin{align*}
\begin{pmatrix}
           b_1 & Ab_1& \cdots A^{n-1}b_1&\cdots& b_m & Ab_m \cdots & A^{n-1}b_m
          \end{pmatrix},
\end{align*}
the first linear independent columns. Then one obtains a list of basis
vectors of the reachability subspace as 
\begin{align*}
 b_1,...,A^{K_1-1}b_1,...,b_m,...,A^{K_m-1}b_m.
\end{align*}
The integers $K_1,...,K_m$ are called the \emph{Hermite indices},
where $K_i:=0$ if the column $b_i$ has not been selected. One has
$K_1+\cdots+K_m=n$ if and only if $(A,B)$ is reachable. The next
result has been proven in \cite{scl2014} for the special case that
$\textbf{P}$ is a single compact interval.

\begin{theorem}\label{MAIN}
Let $\textbf{P}\subset \mathbb{R}$ be a finite union of disjoint
compact intervals. The ensemble of linear systems $\Sigma= \{
(A(\theta),B(\theta))\; |\;\theta \in \mathbf{P} \}$ is uniformly ensemble controllable if the following conditions are satisfied: 
\begin{enumerate}
\item[(i)]  $(A(\theta),B(\theta))$ is reachable for all $\theta\in \textbf{P}$.
\item[(ii)]  The input Hermite indices $K_1(\theta),\ldots, K_m(\theta)$ of
  $(A(\theta),B(\theta))$ are independent of $\theta \in \textbf{P}$.
  

  %
\item[(iii)] For any pair of distinct parameters $\theta,\theta '\in \textbf{P}, \theta\neq \theta '$, the spectra of $A(\theta)$ and $A(\theta')$
  are disjoint:
\begin{equation*}
\sigma (A(\theta))\cap \sigma (A(\theta'))=\emptyset.
\end{equation*}
\item[(iv)] For each $\theta\in \textbf{P}$, the eigenvalues of $A(\theta)$
  have algebraic multiplicity one.
\end{enumerate}
\end{theorem}

\begin{proof}
We show the claim for $\mathbf{P} = \mathbf{P}_1 \cup \mathbf{P}_2$, where $\mathbf{P}_1,\mathbf{P}_2$ are disjoint compact intervals in $\R$. Define $K_i:= \bigcup_{\theta \in \mathbf{P}_i} \sigma(A(\theta))$. The union of $N$ disjoint compact intervals can be concluded by induction. Let $x^*\in C(\mathbf{P},\R^n)$ and $\e>0$ be fixed. According to the proof of Theorem~1 in \cite{scl2014} it is sufficient to prove the assertion for discrete-time single-input ensembles. Then, by  \cite[Theorem~1]{scl2014} the ensembles $\{ (A(\theta),B(\theta))\,|\,\theta \in \mathbf{P}_1 \}$ and $\{ (A(\theta),B(\theta))\,|\,\theta \in \mathbf{P}_2 \}$ are uniformly ensemble controllable. Then, as is shown above there are polynomials $p_1 \in \R[z]$ and $p_2 \in \R[z]$ such that
\begin{align*}
 \sup_{\theta \in \textbf{P}_1} \| p_1(A(\theta))\,b(\theta) - x^*(\theta) \| < \tfrac{\e}{2} \quad \text{ and }  \quad  \sup_{\theta \in \textbf{P}_2} \| p_2(A(\theta))\,b(\theta) - x^*(\theta) \| < \tfrac{\e}{2}.
\end{align*}
Then, by the Stone-Weierstrass Theorem \cite[Theorem~6.6.3]{davis63} there is a polynomial $q \in\R[z]$ satisfying
\begin{align*}
 \sup_{z \in K_1} \| p_1(z) - q(z) \| < \tfrac{\e}{2} \quad \text{ and }  \quad  \sup_{z \in K_2} \| p_2(z) - q(z) \| < \tfrac{\e}{2}.
\end{align*}
Then, for any $\theta \in \mathbf{P}$, w.l.o.g. $\theta \in \mathbf{P}_1$, we have
\begin{align*}
\| q(A(\theta))\,b(\theta) - x^*(\theta) \| \leq \|  q(A(\theta))\,b(\theta)- p_1(A(\theta))\,b(\theta)  \| + \| p_1(A(\theta))\,b(\theta) - x^*(\theta) \| < \e.
\end{align*}
This shows the assertion.
\end{proof}

In the continuous-time case, conditions (i)-(iv)  imply that uniform
ensemble controllability of $\Sigma$ can be achieved in arbitrary time
$T>0$. As pointed out earlier, uniform ensemble controllability is
related to approximation theory. The proof of Theorem~1 in \cite{scl2014} is
based on Mergelyan's Theorem from complex approximation. 

\begin{remark}\label{rem:2}
Condition~(iv) in Theorem~\ref{MAIN} can be replaced by the following two conditions that are sometimes easier to check:
\begin{enumerate}
 \item $A(\theta)$ is diagonalizable by a similarity transformation with uniformly bounded condition number.
 \item $\C \setminus K$ is connected for $K = \bigcup_{\theta \in \mathbf{P}} \sigma(A(\theta))$.
\end{enumerate}
\end{remark}

By Theorem~\ref{prop:1} and Theorem~2.3 in \cite{pamm2014} one obtaines for single-input ensembles the following sufficient condition.

\begin{theorem}\label{OnlyConstantTerm}
Let $\mathbf{P}\subset\mathbb{R}$ be a compact interval and suppose the ensemble $\Sigma=\{(A(\theta),b(\theta))\;|\; \theta \in \textbf{P}\}$ satisfies (i) and (iii) in Theorem~\ref{MAIN}. If there are  $a_1,\dots,a_{n-1}\in\mathbb{R}$ such that for all $\theta \in \mathbf{P}$ the characteristic polynomials are of the form $\chi_{A(\theta)}(z)=z^n-a_{n-1}z^{n-1}-\dots -a_1 z - a_0(\theta)$, then $\Sigma$ is uniformly ensemble controllable.
\end{theorem}

\bigskip

As we next show, there do not exist ensembles of linear systems with more than three parameters that are uniformly ensemble controllable. We need the following lemma from \cite{linnemann1984bestimmung}. 

\begin{lemma}\label{lem:her-const}
 Let $\textbf{P}\subset \R^d$ with
$\overline{\operatorname{int}\,\mathbf{P}} =\mathbf{P}$. 
Let
 $\Pi\colon \R^d \to \R^{n^2+nm}$, $ \theta \mapsto \big(
 A(\theta),B(\theta) \big)$ be real analytic. Then the Hermite indices of the ensemble $\Sigma=\{(A(\theta),B(\theta))\;|\; \theta \in \textbf{P}\}$ are generically constant, i.e. the Hermite indices are constant on an open and dense subset of $\textbf{P}$.
\end{lemma}

\begin{proof}
 Let $K_{n,m}$ denote the set of all $K=(K_1,...,K_m) \in \N_0^m$ with  $K_1+\cdots +K_m \leq n$. For $K \in K_{n,m}$ let $\operatorname{Her} (K) := \{ (A,B)\in \mathbb{R}^{n \times n}\times
  \mathbb{R}^{n\times m} \,| \,\text{ Hermite indices are } K\}$ denote the subset in $\mathbb{R}^{n^2+nm}$ of all systems with Hermite
indices $K$. The set $  \operatorname{Her} (K)$ is a constructible
algebraic subset of $\R^{n^2+nm}$, which induces a disjoint partition
$\bigcup_{K \in K_{n,m}}   \operatorname{Her} (K) =
\R^{n^2+nm}$. Consider the real analytic map $\Pi\colon \R^d  \to
\R^{n^2+nm}$, $ \theta \mapsto \big( A(\theta),B(\theta) \big)$. Then,
the preimage $\Pi^{-1} \big( \operatorname{Her} (K) \big)$ is a constructible analytic subset of $\R^d $ (or the empty
set) and there exists a $K_* \in K_{n,m}$ such
that $\operatorname{dim} \Pi^{-1} \big(\operatorname{Her} (K_*) \big)$
has an interior point. But any constructible analytic subset $S\subset
\mathbb{R}^d$ of dimension $d$ contains an open and dense subset of $\R^d$. Thus the set of interior points of
$\Pi^{-1} \big(\operatorname{Her} (K_*)\big)$ is an open and dense subset of $\R^d$. In
particular, the intersection $\Pi^{-1} \big(\operatorname{Her} (K_*) \big)\cap
\textbf{P}$ contains an open and dense subset 
of $\textbf{P}$. It remains to show that $K_*$  is unique. Suppose there exist two
different $K_1,K_2 \in K_{n,m}$ such that their preimage
$\Pi^{-1}\big( \operatorname{Her} (K_i) \big)$, $i=1,2$ contains open
and dense subsets. Then the intersection $\Pi^{-1}\big(
\operatorname{Her} (K_1) \big) \cap \Pi^{-1}\big( \operatorname{Her}
(K_2) \big)\not= \emptyset$, implying $K_1=K_2$. This shows the assertion.
\end{proof}

\begin{theorem}\label{thm:negative}
Let $\textbf{P}\subset \R^d$ with $\overline{\operatorname{int}\,\mathbf{P}} =\mathbf{P}$. Let $\Pi\colon \R^d \to \R^{n^2+nm}$, $ \theta \mapsto \big(
 A(\theta),B(\theta) \big)$ be real analytic. Suppose the ensemble $\Sigma(\theta)=\{ \big(A(\theta),B(\theta)\big) \,|\, \theta \in \mathbf{P}\}$ is uniformly ensemble controllable.
\begin{enumerate}
 \item[(a)] Then $\operatorname{dim}\,\mathbf{P} \leq 2$.
 \item[(b)] Let $\lambda_1(\theta),...,\lambda_n(\theta)$ denote the
 eigenvalues of $A(\theta)$. Assume that at least one branch  $\{
 \lambda_k(\theta) \,|\, \theta \in \mathbf{P}\}$ of the eigenvalues
 is contained in a one-dimensional real subspace $S$ of $\C$. Then $\operatorname{dim}\,\mathbf{P} =1$.
 \item[(c)] If there is at least one eigenvalue such that $S:=\{ \lambda_k(\theta) \,|\, \theta \in \mathbf{P}\} \subset \R$ then $\operatorname{dim}\,\mathbf{P} =1$.
\end{enumerate}
\end{theorem}

\begin{proof}
As the Hermite indices are generically constant, i.e. independent of the parameter, we may assume that $(A(\theta),B(\theta))$ satisfies the assumptions (i) and (ii) of Theorem~\ref{MAIN}. As the Hermite indices of the family $(A(\theta),B(\theta))$ are independent of parameter $\theta \in \mathbf{P}$, there  exists a continuous family of invertible coordinate transformations $S(\theta)$ such that $(S(\theta)A(\theta)S(\theta)^{-1},S(\theta)B(\theta))$ is in Hermite canonical form. Thus, w.l.o.g. we can assume that $(A(\theta),B(\theta))$ is in Hermite canonical form 
\begin{equation*}\label{eq:Herm}
\begin{pmatrix}
A_{11}(\theta) & \cdots & A_{1m}(\theta)\\
 & \ddots&  \vdots\\
0 & & A_{mm}(\theta)
\end{pmatrix},
\quad
\begin{pmatrix}
b_{1} &  & 0\\
 & \ddots &\\
0 & & b_{m}
 \end{pmatrix}, 
\end{equation*}
where the $m$ single-input subsystems $(A_{kk}(\theta),b_k)\in
\mathbb{R}^{n_k\times n_k} \times \mathbb{R}^{n_k}$ are reachable and
in control canonical form. Note that $b_k$ denotes the first standard
basis vector and thus is independent of $\theta$. Partition the
desired state vector as $x^*(\theta)=\operatorname{col}(x^*_1(\theta)\,\cdots \, x^*_N(\theta))$ with
$x^{*}_{k}(\theta)\in \R^{n_k}$. By Proposition~\ref{lem:sirect-sum} the single-input systems defined by $(A_{kk}(\theta),b_k)$ are uniform ensemble controllable. Applying (E2) for the single input case we conclude that the spectra of $A_{kk}(\theta)$ and $A_{kk}(\theta')$ are disjoint for all distinct parameters $\theta \neq \theta'$ in $\mathbf{P}$. Let $\lambda\colon \mathbf{P} \to \C$ be a locally defined branch of the eigenvalues of $A(\theta)$. Let $S$ denote be real subspace of $\C$ of smallest dimension such that $\lambda(\mathbf{P}) \subset S$. Since the eigenvalues of a matrix depends continuously on the parameters we see that locally $\lambda$ is continuous. By injectivity of $\lambda$ we get $\operatorname{dim}_\R\,\mathbf{P} \leq \operatorname{dim}_\R\,S$. This completes the proof.
\end{proof}

\section{Special classes of ensembles}

Consider an ensemble of harmonic oscillators defined by
\begin{equation*}
 \theta \, A= \theta\, 
\begin{pmatrix}
 0 & -1\\
  1& 0
\end{pmatrix}
,\quad 
B=
\begin{pmatrix}
1& 0\\
0&1
\end{pmatrix},
\end{equation*}
where the parameter is contained in the union $\mathbf{P}$ of finitely many compact intervals with $0 \in \mathbf{P}$.
Note that for $\theta=0$ the matrix $A(0)$ has a double eigenvalue $0$. The Hermite indices are
\begin{equation*}
 \begin{split}
 K_1(\theta)=
 \begin{cases} 
  2 & \text{ if } \theta \neq 0\\
  1 &  \text{ if } \theta = 0
\end{cases}
\quad \text{ and } \quad
 K_2(\theta)=
 \begin{cases} 
  0 & \text{ if } \theta \neq 0\\
  1 &  \text{ if } \theta = 0,
\end{cases}
\end{split}
\end{equation*}
i.e. they are not constant. Thus the ensemble does not satisfy the condition~(ii) in Theorem~\ref{MAIN}. Nevertheless we will show that this family of systems is uniformly ensemble controllable. The next result applies Theorem~\ref{MAIN} to prove an extension of \cite[Theorem~1]{li-cdc-2013}.

\medskip

For fixed matrices $A \in \R^{n \times n}$ and $B\in \R^ {n\times m}$ we consider ensembles of the form
\begin{equation*}\label{eq:sys-li-cont}
\Sigma = \{ (\theta A, B) \, | \, \theta \in \mathbf{P} \}.
\end{equation*}
For this special class of system families the characterization
of uniform ensemble controllability depends on whether the parameter set contains zero or not.

\begin{theorem}\label{thm:li_cdc2013-besser}
Let $\mathbf{P} \subset \R$ be the union of compact intervals. 
Assume that $0 \in \mathbf{P}$. The family $\Sigma=\{ (\theta A,B)\;|\;
\theta \in \mathbf{P}\}$ is uniformly ensemble controllable if and only if $\operatorname{rank}A=n$ and  $\operatorname{rank}B=n$.
\end{theorem}

\begin{proof}
We focus on the continuous-time case; the discrete-time case goes mutatis mutandis. 
Suppose that $\Sigma=\{ (\theta A,B)\;|\; \theta \in \mathbf{P}\}$ is uniformly ensemble controllable. Then, since $0 \in \mathbf{P}$, the necessary condition (E1) implies $\operatorname{rank}B=n$. In particular, we have $m\geq n$. To show the second claim, suppose that  $\operatorname{rank}A<n$. Then zero is an eigenvalue of $A$ and for distinct parameter values $\{\theta_1,...,\theta_{n+1} \} \in \mathbf{P}$ we have $$0 \in \sigma\big(\theta_1A\big) \cap \cdots \cap \sigma\big(\theta_{n+1}A\big)$$ contradicting the necessary condition (E2).

Conversely, assume that  $\operatorname{rank}A=n$ and  $\operatorname{rank}B=n$. Without loss of generality we can assume that
$B=I_n$. The reachability condition (E1) is implied by the
rank condition on $B$. Let $\Lambda$ denote the Jordan canonical form. It is sufficient to consider the ensemble
\begin{equation}\label{eq:sys-cont-theta-lambda}
 \begin{split}
  \tfrac{\partial}{\partial t}x(t,\theta)= \theta \Lambda x(t,\theta)+ I\,u(t).
 \end{split}
\end{equation}
Using Proposition~\ref{lem:sirect-sum} it remains to prove the assertion for one Jordan block. For simplicity we focus on the case of a two dimensional Jordan block; the higher dimensional case follows by an induction argument. Let
\begin{equation}\label{eq:sys-cont-theta-lambda-2}
  \tfrac{\partial}{\partial t}z(t,\theta)= 
  \begin{pmatrix}
 \theta \lambda & \theta\\
  0& \theta \lambda
\end{pmatrix} z(t,\theta)+ 
  \begin{pmatrix}
  1 & 0\\
  0& 1
\end{pmatrix}
u(t).
\end{equation}
The solution to~\eqref{eq:sys-cont-theta-lambda-2} is given by 
\begin{equation*}
 \varphi(T,\theta,u)= 
\int_0^T 
\begin{pmatrix}
\operatorname{e}^{\theta \lambda \, (T-s)}  u_1(s)  + \theta \, (T-s)\operatorname{e}^{\theta \lambda \, (T-s)}  u_2(s)\\
\operatorname{e}^{\theta \lambda \, (T-s)}  u_2(s) 
\end{pmatrix} \di{s}.
\end{equation*}
Given $z^*=\operatorname{col}( z^*_1\, \, z^*_2) \in C(\mathbf{P},\R^2)$ and $\e>0$. By applying Theorem~\ref{MAIN} to the ensemble
\begin{equation*}
  \tfrac{\partial}{\partial t}z_2(t,\theta)= 
  \theta \lambda
\, z_2(t,\theta)+ 
u(t),
\end{equation*}
there is an input function $u_2 \colon [0,T]\to \R$ so that $| z_2^*(\theta) -  \varphi_2(T,\theta,u_2)| < \e $ for all $\theta \in \mathbf{P}$. Let
\begin{align*}
 w^*(\theta):= z_1^*(\theta) 
 - \int_0^T 
 \theta \, (T-s)\operatorname{e}^{\theta \lambda \, (T-s)}  u_2(s) \di{s} \in C(\mathbf{P},\R).
\end{align*}
Following the same reasoning there is an input $u_1\colon [0,T] \to \R$ so that
\begin{align*}
| w^*(\theta) 
 - \int_0^T 
\operatorname{e}^{\theta \lambda \, (T-s)}  u_1(s) \di{s}| < \e.
\end{align*}
Consequently, we have
\begin{equation*}
\sup_{\theta \in \mathbf{P}} \| z^*(\theta) - \varphi(T,\theta,u)\| < \e
\end{equation*}
and we are done.
\end{proof}

Theorem~\ref{thm:li_cdc2013-besser} dealt with the situation $m\geq n$. In the subsequent result we do not make this assumption. We use the notation $\lambda \mathbf{P} :=\{ \lambda\,\theta \,|\, \theta \in \mathbf{P}\}$.
\begin{theorem}\label{thm:li_cdc2013-besser-2}
Let $\mathbf{P}$ be the union of compact real intervals with $0 \not \in \mathbf{P}$. 
\begin{enumerate}
\item[(a)]  If the family $\Sigma=\{ (\theta A,B) \,|\,\theta \in \mathbf{P}\}$ is uniformly ensemble controllable then  $(A,B)$ is controllable and $A$ is invertible.


\item[(b)] Let $(A,B)$ be controllable and let $A$ be invertible and diagonalizable such that $\lambda_k \mathbf{P} \cap \lambda_l  \mathbf{P} = \emptyset$ for all $k\neq l \in \{1,...,r\}$. Then the family $\Sigma=\{ (\theta A,B) \,|\,\theta \in \mathbf{P}\}$ is uniformly ensemble controllable.
 \end{enumerate}
\end{theorem}

\begin{proof}
(a) Let the family $\Sigma=\{ (\theta A,B) \,|\,\theta \in \mathbf{P}\}$ be uniformly ensemble controllable. Then, by (E1)  the pair $(\theta A,B)$ is controllable for every $\theta \in \mathbf{P}$. Using the Kalman matrix we have $(A,B)$ is controllable. To show the second claim, suppose that  $\operatorname{rank}A<n$. Then zero is an eigenvalue of $A$ and for distinct parameter values $\{\theta_1,...,\theta_{m+1} \} \in \mathbf{P}$ we have $$0 \in \sigma\big(\theta_1A\big) \cap \cdots \cap \sigma\big(\theta_{m+1}A\big)$$ contradicting the necessary condition (E2).

 
(b) To show the claim, we verify the sufficient conditions of Theorem~\ref{MAIN}. The reachability of the pair $(A,B)$ and the fact that $0 \neq \mathbf{P}$ implies that $(\theta\,A,B)$ is reachable for every $\theta \in  \mathbf{P}$. 
Due to the fact that $0 \not \in \mathbf{P}$ the Hermite indices of $(\theta A, B)$ are independent of $\theta$. 
Moreover, as $\lambda_k \mathbf{P} \cap \lambda_l  \mathbf{P} = \emptyset$ for all $k\neq l \in \{1,...,r\}$ we have $\sigma\big( \theta A \big) \cap \sigma\big( \theta' A \big) = \emptyset$ for all $\theta \neq \theta' \in \mathbf{P}$. 
The assertion then follows from Remark~\ref{rem:2}.
\end{proof}

\section{Ensembles of networks of linear dynamical systems}

We consider ensembles of networks of $N$ identical
single-input-single-output systems $\Sigma=(A,b,c)$, whose dynamics
are given by
\begin{align*}
\dot x_i(t) &= A\,x_i(t) + b\,v_i(t)\\
y_i(t) &=c\, x_i(t).
\end{align*}
Here $A \in \mathbb{R}^{n \times n}$, $ b \in \mathbb{R}^{n}$ and $c \in \mathbb{R}^{1 \times n}$. 
We assume that $(A,b,c)$ is controllable and observable. The $N$
identical systems
$\Sigma$ are coupled via directed links and the interconnection
structure is
described by a directed graph $\mathcal{G}=({\cal
  V},\mathcal{E})$. Here $\mathcal{V}=\{1,\ldots,N \}$ denotes the set
of vertices, i.e. the $N$ SISO systems, and $\mathcal{E}$ describes
the set of edges, i.e. the couplings. We examine the situation where
the coupling strength is uncertain and is assume to vary over a
compact interval $\mathbf{P}\subset \mathbb{R}_+=(0,\infty)$. The
weighted graph adjacency matrix is given by 
\begin{align*}
\mathcal{K}(\theta)=  
\begin{cases}
 k_{ij}(\theta) & \text{ if } (i,j) \in \mathcal{E}\\
0 & \text{ else},
\end{cases}
\end{align*}
where $k_{ij}\colon \mathbf{P} \to (0,\infty)$ are known continuous functions, $(i,j)\in \mathcal{E}$. Thus, $
 \mathcal{K}_{\mathbf{P}} = \{\mathcal{K}(\theta)\, | \, \theta \in \mathbf{P} \}$ denotes a family of adjacency matrices. We assume that there is an external input $u$ which is broadcasted to the systems $\Sigma$ via the input-to-state interconnection vector $\mathcal{B} \in \mathbb{R}^N$. The setting is illustrated in Figure~\ref{fig:setting}.

\begin{figure}[h]
\centering
\vspace{.5cm}
\begin{tikzpicture}[thick,node distance=2.5cm,>=latex',scale=0.8]
\tikzstyle{block} = [draw, fill=blue!5, rectangle, minimum height=4em, minimum width=4em]
\node [block,node distance=3cm] (nodes) {$\begin{matrix} \Sigma &   &   \\  & \ddots & \\ & & \Sigma \end{matrix}$};
\node [block, below of=nodes] (network) {$\begin{matrix} \mathcal{K}(\theta) & \mathcal{B} \end{matrix}$};
\node (1) [rectangle,yshift=-18.35ex,xshift=4ex]{};
\node (2) [rectangle,yshift=-19ex,xshift=-4ex]{};
\node [right of=network,yshift=-3.25ex,xshift=2ex] (output) {$u$};
\node [left of=network,yshift=-2.5ex,xshift=-2ex] (input) {$ $};
\draw [<-] (1) -- node[name=u] {} (output);
\node (1nach2) [below of=nodes,yshift=-.6ex,xshift=12ex]{}; 
\node (2nach1) [below of= nodes,yshift=-.60ex,xshift=-12ex]{};
\draw[-] (nodes) -| (1nach2) node[right,midway] {};  
\draw[->] (1nach2) |- (network) node[right,midway] {$v $}; 
\draw[-] (network) -| (2nach1) node[left,yshift=1ex] {$ $}; 
\draw[->] (2nach1) |- (nodes) node[near end,above] {};
\end{tikzpicture}
\caption{Block diagram of an ensemble of homogenous networks with parameter-dependent couplings.}
\label{fig:setting}
\end{figure}
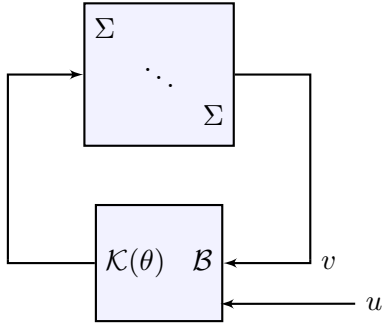

Thus, $\Sigma= \{ (\Sigma,\mathcal{K}(\theta),\mathcal{B}) , \, \theta \in \mathbf{P}\}$ defines an ensemble of networks. Using the state vector $x=\operatorname{col}(x_1\, \cdots \, x_N) \in \mathbb{R}^{nN}$, the network is described by the following control system
\begin{equation}\label{eq:ensemble-net-couplings}
\begin{split}
 \tfrac{\partial}{\partial t} x(t,\theta) &= \big( I \otimes A + \mathcal{K}(\theta)\otimes bc \big) x(t,\theta) + \big(\mathcal{B}\otimes b \big) u(t).
\end{split}
\end{equation}
Let $\mathbf{1}=(1\, \cdots \,1)^\top \in \mathbb{R}^N$ and  $x_i^0 \in \mathbb{R}^n$ denote the initial state of the $i$th node system and $x_0 = \operatorname{col}(x_1^0\, \cdots \,x_N^0) \in \R^{nN}$. The solution to \eqref{eq:ensemble-net-couplings} at time $T>0$ starting in $x_0$  under the interconnection $\mathcal{K}(\theta)$ and the input $u$ is denoted by  $\varphi(T,\mathcal{K}(\theta),x_0,u)$. 

We emphasize that the input $u$ is broadcasted to the systems $\Sigma$
within the network according to the input-to-state vector
$\mathcal{B}$. This phenomenon may be interpreted as that $u$ serves
as an universal input for a whole ensemble of networks that steers the
initial state $x_0$ to a desired terminal state $x^* \in
C(\mathbf{P},\R^{nN})$ in finite time $T$ uniformly for all
interconnection matrices $\mathcal{K}(\theta)$, $\theta \in
\mathbf{P}$. \medskip

The ensemble $\{ (\Sigma,\mathcal{K}(\theta),\mathcal{B})\,|\, \theta
\in \mathbf{P} \}$ is called \emph{robustly synchronizable} to $x^*
\in C(\mathbf{P},\mathbb{R}^{n})$ if for every $\varepsilon >0 $ there
is a $T>0$ and an input-function $u\in L^1([0,T], \mathbb{R})$ such that
\begin{align*}
\sup_{\theta \in \mathbf{P}} \| \varphi(T,\mathcal{K}(\theta),x_0,u) - (\mathbf{1} \otimes  x^*(\theta)) \| < \varepsilon.
\end{align*}

The reason for this terminology is that the desired states to which we
want to control have identical components, i.e.,
\[
\mathbf{1} \otimes  x^*(\theta)= \begin{pmatrix}
x^*(\theta)\\
\vdots\\
x^*(\theta)
\end{pmatrix}\in \R^{nN}.
\]

Thus the single input function achieves synchronization in finite time
$T$, starting from non-synchronuous initial condition 
$
x_0= \operatorname{col} (
x_1(0),
\cdots,
x_N(0)) \in \R^{nN}.
$
We emphasize that the inputs may depend both on the initial condition
$x_0$ and on the family of terminal state $x^*(\theta)$. \medskip
 
As an illustration we consider $N$ identical single-input-single-output (SISO) harmonic oscillators 
\begin{align}\label{def:A-b-c_1}
 A :=  
\begin{bmatrix}
 0 & -1\\
1 & 0
\end{bmatrix},
\quad 
b:=
\begin{bmatrix}
1 \\
0
\end{bmatrix},
\quad 
c:=
\begin{bmatrix}
0 & 1
\end{bmatrix}.
\end{align}
The oscillators are coupled in
a circular manner. The network topology is described by a directed
graph $\mathcal{G}$ with $N$ nodes. The weighted adjacency matrix
${\cal K}(\theta)$ is given by the circulant matrix
\begin{align}\label{eq:ring-adjacency}
 {\cal K}_{\text{ring}}(\theta):=  
\begin{bmatrix}
0 & \theta &   &  \\
  & \ddots & \ddots  & \\
  &  &  \ddots & \theta  \\
\theta & &   & 0 \\
\end{bmatrix}, \quad \theta \in \mathbf{P}:=[\theta^-,\theta^+].
\end{align}
The network is depicted in Figure~\ref{fig:diring}.  

\begin{figure}[h!]
\vspace{.5cm}
\begin{center}
\begin{tikzpicture}[scale=.845]
\def \n {5}
\def \radius {1.95cm}
\def \margin {15} 
\foreach \s in {1,...,\n}
{
  \node at ({360/\n * (\s -1) + 40}:\radius + .25cm) {$\theta$};
  \node[draw, circle] at ({360/\n * (\s - 1)}:\radius) {$\Sigma$};
  \draw[<-, >=latex] ({360/\n * (\s - 1)+\margin}:\radius) 
    arc ({360/\n * (\s - 1)+\margin}:{360/\n * (\s)-\margin}:\radius);
}
\end{tikzpicture}
\end{center}
\caption{Ensemble of rings of $5$ identical harmonic oscillators.}
\label{fig:diring}
\end{figure}
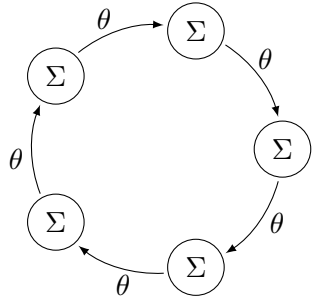
Without loss of generality, assume that the harmonic oscillators are
numbered be such that the external input is applied to the first oscillator. Thus, the input-to-state interconnection vector is $\mathcal{B}=e_1=(1,0,...,0)^\top$. The dynamics of the overall network of systems is of the form
\begin{align}\label{eq:diring-nw}
\begin{split}
 \tfrac{\partial}{\partial t} x(t,\theta) &= \big( I \otimes A + {\cal K}_{\text{ring}}(\theta)\otimes bc\big) x(t) + \big(e_1\otimes b \big) u(t)\\
 x(0,\theta)&= x_0\in \mathbb{R}^{2nN}.
\end{split}
\end{align}
Let $x^* \in C(\mathbf{P},\mathbb{R}^2)$ denote the desired terminal
states of the harmonic oscillators. 
\medskip

\begin{proposition}\label{prop:hos} Let $\textbf{P}$ denote a
compact interval in $(0,1)\cup (1,\infty)$. 
The circular network ensemble of harmonic oscillators (\ref{eq:diring-nw})
is robustly synchronizable to $\mathbf{1}\otimes x^*$. 
\end{proposition}

\medskip
\begin{proof}
Let $\omega:=\operatorname{e}^{2\pi i\tfrac{1}{N}}$ denote the
primitive $N$-th root of unity. The adjacency matrix is a circulant matrix with spectrum  
\begin{align*}
\sigma({\cal K}(\theta)) = \left\{ \theta  \operatorname{e}^{2\pi i\tfrac{l}{N}}  \, \big|\, l=0,...,N-1\right\}.
\end{align*}
The family $\{ {\cal K}(\theta)\,|\, \theta \in \mathbf{P}\}$ of
circulant matrices is simultaneously diagonalizable using the unitary 
Vandermonde matrix
\begin{align}\label{eq:S-circ}
S=
\begin{bmatrix}
1 & 1 & \cdots  & 1 \\
1  & \omega & \cdots  & \omega^{N-1}\\
\vdots &  \vdots & \cdots & \vdots  \\[1ex]
1 & \omega^{N-1}&  \cdots & \omega^{(N-1)^2}  \\
\end{bmatrix}.
\end{align}
Applying the change of coordinates $S\otimes I$ to the network
dynamics yields the state space equivalent system
\begin{equation}\label{ring-diag}
\begin{split}
\tfrac{\partial}{\partial t} {z}(t,\theta)& = 
\operatorname{diag}\left(
A+\lambda_l(\theta)bc \right)
x(t) + 
\left( \mathbf{1} \otimes b \right) u(t)\\
 z(0,\theta)&= (S \otimes I ) x_0.
\end{split}
\end{equation}
Since $S\mathbf{1} = Ne_1$, the corresponding desired terminal states are $Ne_1 \otimes x^*(\theta) \in C(\mathbf{P},\mathbb{R}^{2N})$. We show that Theorem~\ref{MAIN} can be applied to \eqref{ring-diag}.
A simple calculation shows that
\begin{align}\label{spec-ring-trafo}
\bigcup_{l=1}^N \sigma \left(
A+\lambda_l(\theta)bc \right)= 
\bigcup_{l=1}^N \left\{ w \in \mathbb{R} \,  \Big|\,  w^2 - (\theta  \operatorname{e}^{2\pi i\tfrac{l}{N}} -1) =0   \right\}.
\end{align}
Thus
\begin{align*}
 \sigma( A + \lambda_l(\theta)bc) \cap  \sigma( A + \lambda_k(\theta)bc) = \emptyset 
\end{align*}
holds for all $\theta \in \mathbf{P}$ and $l \neq k \in
\{1,...,N\}$. Moreover, since $(A+\lambda_l(\theta)bc,b)$ is
controllable for all $\theta \in \mathbf{P}$ and $l \in \{1,...,n\}$,
we conclude that the parallel connected system \eqref{ring-diag} is
controllable. Furthermore, by inspection from \eqref{spec-ring-trafo},
conditions~(iii) and (iv) in Theorem~\ref{MAIN} are satisfied. This
completes the proof. 
\end{proof}

\medskip

Proposition \label{prop:hos} corrects an error in \cite{mtns2014},
where the corresponding result was claimed for undirected circular
graphs, with the symmetric circulant adjacency matrix
\begin{align}\label{symcir} 
\begin{bmatrix}
0 & \theta &   & \theta \\
 \theta & \ddots & \ddots  & \\
  &\ddots &  \ddots & \theta  \\
\theta & &  \theta & 0 \\
\end{bmatrix}.
\end{align}
This matrix has real eigenvalues $\theta \cos\left({2\pi
    \tfrac{l}{N}}\right)\,,\, l=0,...,N-1$. Thus, for $N=4$, the
    necessary condition (E2) is not satisfied as
    $\cos(\tfrac{\pi}{2})= \cos(\tfrac{3\pi}{2})=0$. Therefore, the
    ring of oscillators with symmetric coupling matrix (\ref{symcir})
    is not robust synchronizable.\bigskip


Next, we discuss a scenario where the node systems depend on a
single parameter and are interconnected by a fixed graph adjacency
matrix $\mathcal{K}$. That is, let
$\Sigma(\theta)=(A(\theta),b(\theta),c(\theta))$ be an ensemble of SISO  systems,
%
%
where $A\in C(\mathbf{P},\mathbb{R}^{n \times n})$, $ b \in
C(\mathbf{P},\mathbb{R}^{n})$ and $c \in C(\mathbf{P},\mathbb{R}^{1
  \times n})$. We assume that the system
$(A(\theta),b(\theta),c(\theta))$ is controllable and observable for
every $\theta \in \mathbf{P}$. Let $\mathcal{G}=({\cal
  V},\mathcal{E})$ be a directed graph with a fixed weighted adjacency matrix $\mathcal{K} \in \mathbb{R}^{N \times N}$. 

\begin{figure}[h!]
\centering
\vspace{.5cm}
\begin{tikzpicture}[thick,node distance=2.5cm,>=latex',scale=0.8]
\tikzstyle{block} = [draw, fill=blue!5, rectangle, minimum height=4em, minimum width=4em]
\node [block,node distance=3cm] (nodes) {$\begin{matrix} \Sigma(\theta) &   &   \\  & \ddots & \\ & & \Sigma(\theta) \end{matrix}$};
\node [block, below of=nodes] (network) {$\begin{matrix} \mathcal{K} & \mathcal{B}  \end{matrix}$};
\node (1) [rectangle,yshift=-18.35ex,xshift=4ex]{};
\node (2) [rectangle,yshift=-19ex,xshift=-4ex]{};
\node [right of=network,yshift=-3.25ex,xshift=3ex] (output) {$u$};
\node [left of=network,yshift=-2.5ex,xshift=-2ex] (input) {$ $};
\draw [<-] (1) -- node[name=u] {} (output);
\node (1nach2) [below of=nodes,yshift=-0.6ex,xshift=12ex]{}; 
\node (2nach1) [below of= nodes,yshift=-0.60ex,xshift=-12ex]{};
\draw[-] (nodes) -| (1nach2) node[right,midway] {};  
\draw[->] (1nach2) |- (network) node[right,midway] {$v$}; 
\draw[-] (network) -| (2nach1) node[left,yshift=1ex] {}; 
\draw[->] (2nach1) |- (nodes) node[near end,above] {};
\end{tikzpicture}
\caption{Block diagram of the ensemble of homogenous networks with parameter-dependent node systems.}
\label{fig:setting}
\end{figure}
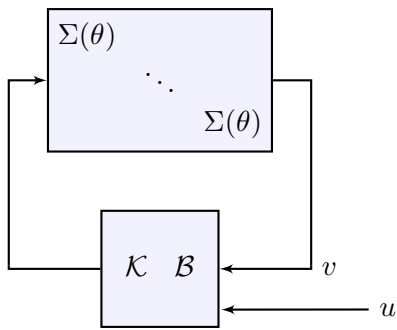

Apply a single, parameter-independent external input $u$ to the
systems $\Sigma(\theta)$, using an input-to-state interconnection
vector $\mathcal{B} \in \mathbb{R}^N$. Let $x=\operatorname{col}(x_1\,
\cdots \, x_N) \in \mathbb{R}^{nN}$ denote the sate vector of the
network. The ensemble of networks we are interested in is described by the control system
\begin{equation}\label{eq:ensemble-net-node-p}
 \tfrac{\partial}{\partial t} x(t,\theta) = \big( I \otimes A( \theta) + \mathcal{K}\otimes b(\theta)c(\theta) \big) x(t,\theta) + \big(\mathcal{B}\otimes b(\theta) \big) u(t).
\end{equation}
We note that \eqref{eq:ensemble-net-node-p} is a special case of the
more general setting described in \cite{fuhrmann2012reachability}. In
particular, the controllability properties of the network are easily
established. 
In fact, (\ref{eq:ensemble-net-node-p}) is controllable if and only if
$({\cal K},{\cal B})$ and $(A(\theta),b(\theta),c(\theta))$ is
controllable for all $\theta$; cf. \cite{fuhrmann2012reachability}. In
order to simplify the analysis, we assume that the adjacency matrix
${\cal K}$ has distinct eigenvalues $\lambda_1,...,\lambda_N$, which
implies that there is an invertible matrix $S$ with $S{\cal K}S^{-1} = {\cal D}
=\operatorname{diag}(\lambda_1,...,\lambda_N)$.  Using the change of
coordinates $z=(S \otimes I)x$, system \eqref{eq:ensemble-net-node-p}
is state-space equivalent to
\begin{equation}\label{eq:ensemble-net-node}
\begin{split}
 \tfrac{\partial}{\partial t} z(t,\theta) &= \big( I \otimes A( \theta) + \mathcal{D}\otimes b(\theta)c(\theta) \big) z(t,\theta) + \big(S\mathcal{B}\otimes b(\theta) \big) u(t)
\end{split}.
\end{equation}
Note that $\lambda_k\neq \lambda_l$ for all $k\neq l \in
\{1,....,N\}$. System (\ref{eq:ensemble-net-node}) is equivalent to
\begin{align*}
 \tfrac{\partial}{\partial t} z(t,\theta) =  
\begin{pmatrix}
 A( \theta) + \lambda_1 b(\theta)c(\theta)  &  & \\
    & \ddots &   \\
     & & A( \theta) + \lambda_N b(\theta)c(\theta) 
\end{pmatrix} z(t,\theta)
+ 
\big( S{\cal B} \otimes b \big)
u(t).
\end{align*}
Thus the spectrum of $ I \otimes A( \theta) + \mathcal{K}\otimes b(\theta)c(\theta)$  is
\begin{align*}
 \sigma \left(  A( \theta) + \lambda_1  \, b(\theta)c(\theta) \right) \cup \cdots \cup  \sigma \left(  A( \theta) + \lambda_N\, b(\theta)c(\theta) \right).
\end{align*}
Since $(A(\theta),b(\theta),c(\theta))$ is controllable and observable
we obtain  
\begin{align*}
 \operatorname{det} \left( zI- \big(A(\theta) - \lambda_k\, b(\theta)c(\theta) \big) \right) = q_\theta (z) + \lambda_k \, p_\theta (z)
\end{align*}
with coprime polynomials $p_\theta(z), q_\theta(z)$. This shows, for
fixed $\theta \in \mathbf{P}$ and $\lambda_k \neq \lambda_l$, that there is no $z \in \mathbb{C}$ such that
\begin{align*}
 q_\theta (z) + \lambda_k\,  p_\theta (z) =  q_\theta (z) + \lambda_l \, p_\theta (z) .
\end{align*}
Therefore, one can apply Theorem~\ref{MAIN}, which proves the next result.

\begin{proposition}
Let $\mathbf{P}\subset \R$ be compact. Suppose the system
$\Sigma(\theta)=(A(\theta),b(\theta),c(\theta))$ is controllable and observable for every $\theta \in \mathbf{P}$. Assume that the pair $({\cal K},{\cal B}) \in \R^{n \times (n+1) }$ is controllable and ${\cal K}$ is diagonalizable.  Then the ensemble of networks \eqref{eq:ensemble-net-node-p} is robustly synchronizable if 
the eigenvalues of $A( \theta) + \lambda  \, b(\theta)c(\theta) $ have algebraic multiplicity one for each $\theta \in \mathbf{P}$ and $\lambda \in \sigma({\cal K})$ and for each $\theta \neq \theta' \in \mathbf{P}$ and $\lambda_l \neq \lambda_k \in \sigma({\cal K})$ it holds
\begin{align*}
 \sigma \left(  A( \theta) + \lambda_l  \, b(\theta)c(\theta) \right) \cap  \sigma \left(  A( \theta') + \lambda_k\, b(\theta')c(\theta') \right) = \emptyset.
\end{align*}

\end{proposition}

\section{Conclusions}

An approximation theoretical approach to controlling state ensembles of linear systems has been proposed. We concluded that the ensemble control problem for continuous-time and discrete-time systems is equivalent. For analytic parameter-dependent linear systems it is shown that uniform ensemble controllability is not possible for more than three parameters. A complete characterization of uniform ensemble controllability is presented for a special class of ensembles. An application to robust synchronization using broadcasted open-loop controls is given. Our approach is based on information of the spectrum of the system matrices (see assumption (iv) in Theorem~\ref{MAIN}). An open problem is to find relaxed necessary and sufficient conditions using the concept of pseudospectra.


\section*{Acknowledgements}

This research is partially supported  by the grant HE
1858/13-1 from the German Research Foundation. Parts of this work has been presented at the 8th Workshop
Structural Dynamical Systems: Computational Aspects (SDS2014). We thank the organizers for this beautiful workshop.


\end{document}